\documentclass[a4paper,12pt,reqno]{amsart}
\usepackage[T1]{fontenc}
\usepackage[utf8]{inputenc}
\usepackage[libertinus,vvarbb]{newtx}
\usepackage[margin=1in]{geometry}
\usepackage{enumitem}
\usepackage{graphicx}
\usepackage[svgnames]{xcolor}
\usepackage{xparse}
\usepackage{xurl}
\usepackage[bookmarksdepth=2]{hyperref}

\hypersetup{colorlinks=true,urlcolor=MidnightBlue,linkcolor=MidnightBlue,citecolor=MidnightBlue}

\usepackage[normalem]{ulem}

\newcommand{\stkout}[1]{\ifmmode\text{\sout{\ensuremath{#1}}}\else\sout{#1}\fi}
\newcommand{\crsout}[1]{\ifmmode\text{\xout{\ensuremath{#1}}}\else\xout{#1}\fi}

\numberwithin{equation}{section}

\newtheorem{theorem}{Theorem}[section]
\newtheorem{proposition}[theorem]{Proposition}
\newtheorem{lemma}[theorem]{Lemma}

\theoremstyle{remark}
\newtheorem{remark}[theorem]{Remark}

\DeclareMathOperator{\re}{Re}

\newcommand{\ind}{\mathbb{1}}

\newcommand{\pr}{\mathbb{P}}
\newcommand{\ex}{\mathbb{E}}
\newcommand{\C}{\mathbb{C}}
\newcommand{\R}{\mathbb{R}}

\newcommand{\fourier}{\mathscr{F}}
\newcommand{\ph}{\varphi}

\mathchardef\mathhyphen="2D

\newcommand{\cm}{\mathscr{C}\mspace{-4mu}\mathscr{M}}
\newcommand{\am}{\mathscr{A}\mspace{-4mu}\mathscr{M}}
\newcommand{\amcm}{\am\mspace{-2mu}\mathhyphen\mspace{-1mu}\cm}

\renewcommand{\le}{\leqslant}
\renewcommand{\ge}{\geqslant}

\NewDocumentCommand{\formula}{ssom}{%
 \IfBooleanTF{#1}{%
  \IfBooleanTF{#2}{%
   \IfValueTF{#3}%
    {\begin{align}\label{#3}\begin{gathered}#4\end{gathered}\end{align}}%
    {\begin{gather}#4\end{gather}}%
  }{%
   \IfValueTF{#3}%
    {\begin{align}\label{#3}\begin{aligned}#4\end{aligned}\end{align}}%
    {\begin{gather*}#4\end{gather*}}%
  }%
 }{%
  \IfValueTF{#3}%
   {\begin{align}\label{#3}#4\end{align}}%
   {\begin{align*}#4\end{align*}}%
 }%
}

\newcommand{\doi}[1]{\href{https://doi.org/#1}{\textsf{\scriptsize DOI:#1}}}

\NewDocumentCommand{\link}{oom}{\href{#3}{\textsf{\scriptsize \IfValueTF{#1}{#1}{#3}}}\IfValueT{#2}{{\textsf{\scriptsize #2}}}}

\begin{document}

\title{Bell-shaped functions are Poisson kernels}
\author{Mateusz Kwaśnicki}
\thanks{Work supported by the National Science Centre, Poland, grant no.\@ 2023/49/B/ST1/04303}
\address{Mateusz Kwaśnicki \\ Department of Analysis and Stochastic Processes \\ Wrocław University of Science and Technology \\ Wybrzeże Wyspiańskiego 27 \\ 50-370 Wrocław, Poland}
\email{\href{mailto:mateusz.kwasnicki@pwr.edu.pl}{\textsf{mateusz.kwasnicki@pwr.edu.pl}}}
\keywords{Elliptic operator, harmonic function, Poisson kernel, extension technique, bell-shape}
\subjclass[2020]{%
 26A51, 
 35J25, 
 35J70, 
 60J60
}

\begin{abstract}
In a recent paper [Kwaśnicki, \textit{Poisson kernels on the half-plane are bell-shaped,} Bull.\@ London Math.\@ Soc.\@ 58 (2026), no.~e70303], the author proved that Poisson kernels for a class of elliptic operators in the half-plane are bell-shaped. By allowing one-sided barriers, we prove a converse result.
\end{abstract}

\maketitle

%
%

\section{Introduction and statement of the result}

Theorem~1.1 in~\cite{k26} states that Poisson kernels for second-order elliptic operators in the half-plane, with coefficients depending only on the vertical variable, are weakly bell-shaped. The converse question was discussed in Remark~1.2 therein. The purpose of this note is to explain how allowing one-sided barriers in the vertical variable leads to an answer to this question.

Recall that a nonnegative smooth function $f$ on $\R$ is \emph{bell-shaped} if it converges to zero at $\pm\infty$ and its $n$th derivative changes sign exactly $n$ times for $n = 0, 1, 2, \ldots$\, A finite nonnegative measure is \emph{weakly bell-shaped} if all its convolutions with the Gauss--Weierstrass kernel are bell-shaped. A smooth weakly bell-shaped function is automatically bell-shaped.

The probabilistic formulation of the main result of~\cite{k26} involves planar \emph{shift-invariant diffusions}, studied in~\cite{k23}. These are continuous strong Markov processes $(X_t, Y_t)$ in $\R \times [0, R)$, where $R \in (0, \infty]$, whose law is invariant under horizontal translations. In~\cite{k23, k26}, $Y_t$ is assumed to be transitive on $[0, R)$. (By mistake, this condition is not included in Definition~1.1 in~\cite{k23}, but it is clearly assumed in the proofs, and discussed in detail in Section~3 therein.) Here we allow $Y_t$ to be intransitive, and in particular we introduce one-sided barriers that $Y_t$ can only cross in one direction.

Let $T_0$ be the hitting time of $\R \times \{0\}$. Under the transitivity assumption, Corollary~1.3 in~\cite{k26} states that the distribution of $X(T_0)$ is weakly bell-shaped. By dropping the transitivity condition, we prove the following converse statement.

\begin{theorem}
\label{thm:main}
Every weakly bell-shaped subprobability measure on $\R$ is the hitting distribution for a possibly intransitive shift-invariant diffusion in $\R \times [0, \infty)$.
\end{theorem}

By hitting distribution we mean the law of $X(T_0)$. Note that we allow $(X_t, Y_t)$ to be killed at a finite lifetime, and in this case the law of $X(T_0)$ is a subprobability measure: the defective distribution of this random variable on the event $\{T_0 < \infty\}$.

\begin{remark}
In this work we focus on the probabilistic construction. Let us, however, briefly comment on the analytic counterpart.

The generators of transitive shift-invariant diffusions $(X_t, Y_t)$, studied in~\cite{k23, k26}, have the formal expression
\formula{
 L & = \tfrac{1}{2} a(y) \partial_{xx} + b(y) \partial_{xy} + \tfrac{1}{2} c(y) \partial_{yy} + d(y) \partial_x + e(y) \partial_y ,
}
with nonnegative second-order part. Here the coefficients can be measures or even Schwartz distributions, but a detailed treatment of regularity issues is complicated and it is omitted in~\cite{k23}. Instead, that work focuses on \emph{regular} transitive shift-invariant diffusions, which additionally require that $X_t$ is a local martingale, and $Y_t$ is the Brownian motion in $[0, R)$, reflected at $0$ and killed at $R$ (if $R < \infty$). In this case the generator is given by
\formula{
 L & = \tfrac{1}{2} (a(dy) + (b(y))^2 dy) \partial_{xx} + b(y) \partial_{xy} + \tfrac{1}{2} \partial_{yy} ,
}
where $a(dy)$ is a nonnegative locally finite measure on $[0, R)$ and $b(y)$ is a locally square-integrable function on $[0, R)$. For a rigorous discussion of the definition of $L$, we refer to~\cite{k23, k26}.

Generators of intransitive shift-invariant diffusions are given by the same formal expressions, but their domains can be different. A two-way interface between reflecting layers, as in Section~\ref{sec:amcm}, corresponds to matching of values and normal derivatives. A one-way interface above a reflected layer, as in the exponential layer in Section~\ref{sec:pff}, corresponds to a one-sided Neumann condition at its upper boundary, together with matching of values from the layer above. An interface above a layer with a deterministic downward motion, as in the Gaussian layer also in Section~\ref{sec:pff}, corresponds to matching the values only; no Neumann condition is imposed there.

The hitting distribution is the Poisson kernel for the corresponding generator $L$. Thus, Theorem~\ref{thm:main} states that every weakly bell-shaped probability measure is a Poisson kernel for an extended class of operators $L$.
\end{remark}

\begin{remark}
Theorem~\ref{thm:main} and the results of~\cite{k26} strongly suggest the following characterisation: $f$ is a weakly bell-shaped distribution if and only if it is the (nontrivial) hitting distribution of an intransitive shift-invariant diffusion in $\R \times [0, \infty)$. Theorem~\ref{thm:main} proves the direct implication, while Theorem~1.1 in~\cite{k26} gives the converse for transitive shift-invariant diffusions. The complete converse result can be approached as follows.

It is intuitively clear that every intransitive diffusion can be obtained by stacking a transitive layer atop a family of intransitive layers with one-way interfaces, in a way similar to the construction carried out in Section~\ref{sec:pff}, but possibly much more complicated. The results of~\cite{k26} imply that each such layer, crossed from its top boundary to the bottom one, contributes a Pólya frequency function to the overall hitting distribution. Finally, the convolution with a Pólya frequency function preserves the class of bell-shaped distributions, and this class is also closed by (nontrivial) vague limits.

We do not pursue this direction here: a rigorous version of the above argument would take us too far from the main scope of this note.
\end{remark}

Our proof follows the factorisation of weakly bell-shaped functions into a Pólya frequency factor and an absolutely monotone-then-completely monotone factor, obtained in~\cite{ks22}. The latter factor is realised above the starting level, using the harmonic extension technique of~\cite{k22,k23}. The former is obtained by inserting a sequence of exponential layers below the starting level. One-sided barriers ensure that the successive horizontal displacements are independent, so that their distributions are convolved.

The Fourier transform of a finite measure $f$ is denoted by
\formula{
 \fourier f(\xi) & = \int_\R e^{-i \xi x} f(dx) .
}
The hat symbol has a different usage in this paper, and so we do not abbreviate $\fourier f$ as $\hat f$. We denote by $\pr^{x, y}$ and $\ex^{x, y}$ the probability and expectation for the process $(X_t, Y_t)$ started at $(X_0, Y_0) = (x, y)$, and we simply write $\pr^y$ and $\ex^y$ when $x = 0$.

%
%

\section{The spectral problem}

Suppose that $(X_t, Y_t)$ is a regular transitive shift-invariant diffusion. Proposition~5.1 in~\cite{k26} identifies the characteristic function of $X(T_0)$ with the solution of the ODE
\formula[eq:ode]{
 \tfrac{1}{2} \ph_\xi''(dy) & = \tfrac{1}{2} \xi^2 \ph_\xi(y) a(dy) + \tfrac{1}{2} \xi^2 (b(y))^2 \ph_\xi(y) dy - i \xi b(y) \ph_\xi'(y) dy .
}
Namely, we have
\formula[eq:poisson]{
 \ex^y \bigl[ e^{i \xi X(T_0)} \ind_{\{T_0 < \infty\}} \bigr] & = \ph_\xi(y) .
}
We choose the solution $\ph_\xi$ which is a bounded function satisfying $\ph_\xi(0) = 1$, and if $R$ is finite, then we additionally require $\ph_\xi(R^-) = 0$. We refer to Theorem~3.1 in~\cite{k26} for a detailed statement and further discussion.

By the results of~\cite{k23}, the boundary trace of $(X_t, Y_t)$ is a Lévy process on $\R \times \{0\}$, which we identify with $\R$, with characteristic exponent
\formula{
 \psi(\xi) & = -\tfrac{1}{2} \ph_\xi'(0^+) + \tfrac{1}{2} a(\{0\}) \xi^2 .
}
Note that~\cite{k23} uses a slightly different normalisation, see Theorem~3.1 in~\cite{k26} for the above form. We will not need the above result here, but the above function $\psi$ appears in some expressions below. 

%
%

\section{Absolutely monotone-then-completely monotone layer}
\label{sec:amcm}

A measure $g$ is \emph{absolutely monotone-then-completely monotone}, or $\amcm$ in short, if it has an absolutely monotone density on $(-\infty, 0)$ and a completely monotone density on $(0, \infty)$, with an additional nonnegative atom at $0$ allowed. A continuous function $\psi : \R \to \C$ is a \emph{Rogers function} if $\psi(-\xi) = \overline{\psi(\xi)}$, and $\psi$ extends to a holomorphic function on the right complex half-plane $\re \xi > 0$, satisfying $\re(\psi(\xi) / \xi) \ge 0$ there. We refer to~\cite{k19, k20, k26} for further discussion.

We will need the following converse of Proposition~2.4 in~\cite{k26}.

\begin{lemma}
\label{lem:reciprocal}
If $g$ is a nonzero $\amcm$ subprobability measure, then
\formula{
 \hat \psi(\xi) & = \frac{1}{2} \biggl( \frac{1}{\fourier g(-\xi)} - 1 \biggr)
}
is a Rogers function.
\end{lemma}

\begin{proof}
By Bernstein's theorem (see Corollary~3.3 in~\cite{k20}), for real $\xi$ we have
\formula[eq:mixture]{
 \fourier g(-\xi) & = p + \int_{(0, \infty)} \frac{s}{s - i \xi} \, \mu_+(ds) + \int_{(0, \infty)} \frac{s}{s + i \xi} \, \mu_-(ds) ,
}
where $p = g(\{0\}) \ge 0$ and $\mu_+$ and $\mu_-$ are finite nonnegative measures. The right-hand side of~\eqref{eq:mixture} extends to a holomorphic function $G(\xi)$ on $\C \setminus i \R$. If $\re \xi > 0$, we have
\formula{
 \re \frac{s \xi}{s \pm i \xi} & = \frac{s^2 \re \xi}{|s \pm i \xi|^2} > 0 .
}
Consequently, $\re(\xi G(\xi)) > 0$. In particular, $G$ has no zeroes in the right half-plane, and
\formula{
 \re\frac{1}{\xi G(\xi)} & > 0 .
}
Thus, $\psi(\xi) = 1 / G(\xi)$ is a Rogers function. By continuity, $G(0^+) = \fourier g(0^+) = g(\R) \le 1$, and hence $\psi(0^+) \ge 1$, that is, the constant term in the integral representation of $\psi$ is at least $1$ (see Theorem~3.3 and Remark~3.4(f) in~\cite{k19}). We conclude that $\hat \psi(\xi) = \tfrac{1}{2} (\psi(\xi) - 1)$ is a Rogers function.
\end{proof}

By~\cite{k23} (see also Theorem~3.1 in~\cite{k26}), $\hat \psi$ is the characteristic exponent of the boundary trace of an appropriate regular transitive shift-invariant diffusion $(\hat X_t, \hat Y_t)$ in $\R \times [0, \hat R)$. We denote the corresponding coefficients by $\hat a(dy)$ and $\hat b(y)$.

We now apply the construction fully analogous to the one given in Sections~3.3 and~3.4 in~\cite{k26}: we stack the layer $(\hat X_t, \hat Y_t)$ in $\R \times [0, \hat R)$ atop a trivial layer in $\R \times [0, 1]$, where the horizontal component remains constant. The interface between these layers is two sided, unless the coefficient $\hat a(dy)$ has an atom at $0$.

A formal construction is as follows. We put $R = 1 + \hat R$, and we define the coefficients $a(dy)$ and $b(y)$ of $(X_t, Y_t)$ by
\formula{
 a(dy) & = \hat a(dy - 1) && \text{on } [1, R), & a(dy) & = 0 & & \text{on } [0, 1) , \\
 b(y) & = \hat b(y - 1) && \text{on } [1, R), & b(y) & = 0 && \text{on } [0, 1) .
}
The solution $\ph_\xi(y)$ of the corresponding ODE~\eqref{eq:ode} is affine in $[0, 1]$, proportional to $\hat \ph_\xi(y - 1)$ in $[1, R)$, and the matching condition at $y = 1$ is
\formula{
 \tfrac{1}{2} (\ph_\xi'(1^+) - \ph_\xi'(1^-)) & = \tfrac{1}{2} \ph_\xi''(\{1\}) = \tfrac{1}{2} \ph_\xi(1) a(\{1\}) \xi^2 .
}
Since $\ph_\xi$ is affine on $[0, 1]$, we have $-\tfrac{1}{2} \ph_\xi'(1^-) = \tfrac{1}{2} (1 - \ph_\xi(1))$. On the other hand, $\ph_\xi(y) = \ph_\xi(1) \hat \ph_\xi(y - 1)$ for $y \in [1, R)$, and so
\formula{
 -\tfrac{1}{2} \ph_\xi'(1^+) + \tfrac{1}{2} \ph_\xi(1) a(\{1\}) \xi^2 & = -\tfrac{1}{2} \ph_\xi(1) \hat \ph_\xi'(0^+) + \tfrac{1}{2} \ph_\xi(1) \hat a(\{0\}) \xi^2 = \ph_\xi(1) \hat \psi(\xi) .
}
It follows that $\tfrac{1}{2} (1 - \ph_\xi(1)) = \ph_\xi(1) \hat \psi(\xi)$, that is,
\formula[eq:hat:check]{
 \ph_\xi(1) & = \frac{1}{1 + 2 \hat \psi(\xi)} \, .
}
Combining this with~\eqref{eq:poisson} and Lemma~\ref{lem:reciprocal}, we conclude that
\formula{
 \ex^1 \bigl[ e^{i \xi X(T_0)} \ind_{\{T_0 < \infty\}} \bigr] & = \ph_\xi(1) = \frac{1}{1 + 2 \hat \psi(\xi)} = \fourier g(-\xi) ,
}
that is, $X(T_0)$ has distribution $g$. This proves the following result.

\begin{proposition}
\label{prop:amcm}
Every nonzero $\amcm$ subprobability measure is the hitting distribution of a regular transitive shift-invariant diffusion in $\R \times [0, R)$, where $R \in (0, \infty]$.
\end{proposition}

We call the diffusion constructed above the \emph{$\amcm$ layer.} Note that reflection at $\R \times \{0\}$ is irrelevant for the distribution of $X(T_0)$, studied in~\cite{k26}. It only played a role in the definition of the boundary trace in~\cite{k23}.

\begin{remark}
Equation~\eqref{eq:hat:check} also follows directly from the construction carried out in Section~3.3 of~\cite{k26}, at least when $\hat a(\{0\}) = 0$. With $\check R = 1$, $\check a(dy) = 0 dy$ and $\check b(y) = 0$, we then have $\check \psi(\xi) = \tfrac{1}{2}$, and Eq.~(3.7) in~\cite{k26} gives~\eqref{eq:hat:check}. The case $\hat a(\{0\}) > 0$ requires a minor tweak: one has to move this atom from $\hat a$ to $\check a$.
\end{remark}

%
%

\section{Pólya frequency function layer}
\label{sec:pff}

A probability measure $h$ on $\R$ is a \emph{Pólya frequency function} if it is the convolution of a (possibly degenerate) Gaussian distribution and centred exponential distributions (or their reflections). The number of centred exponential factors may be infinite, but in this case the series of corresponding variances necessarily converges. Thus, the characteristic function of $h$ is
\formula[eq:pff]{
 \fourier h(-\xi) & = e^{i p \xi - q \xi^2 / 2} \prod_n \frac{e^{-i r_n \xi}}{1 - i r_n \xi} ,
}
where $p \in \R$ and $q \in [0, \infty)$ are the mean and the variance of the Gaussian component, while $r_n \in \R \setminus \{0\}$ are the parameters of the centred exponential factors (with $r_n < 0$ corresponding to the reflected exponential distribution). The product in~\eqref{eq:pff} may be finite or infinite. We refer to~\cite{k20} for further discussion.

\subsection{Gaussian layer}

Fix $p \in \R$ and $q \in (0, \infty)$, and consider the intransitive shift-invariant diffusion $(X_t, Y_t)$ in $\R \times [0, q]$ such that $X_t$ is the standard Brownian motion and $Y_t = \max\{0, Y_0 - t\}$. Clearly, $T_0 = Y_0$, and so
\formula{
 \ex^{p, q} \bigl[ e^{i \xi X(T_0)} \bigr] & = \ex^{p, q} \bigl[ e^{i \xi X(q)} \bigr] = e^{i p \xi - q \xi^2 / 2} .
}
Thus, the Gaussian distribution with mean $p$ and variance $q$ is the Poisson kernel of an intransitive shift-invariant diffusion in $\R \times [0, q]$, called the \emph{Gaussian layer.} Note that the hitting time $T_0$ is deterministic: under $\pr^{p, q}$ we have $T_0 = q$.

\subsection{Exponential layer}

Fix $r \in \R \setminus \{0\}$, set $s = r^2$, and let $(X_t, Y_t)$ be the transitive regular shift-invariant diffusion in $\R \times [0, \infty)$ corresponding to coefficients $a(dy) = 0 dy$ and $b(y) = r^{-1} \ind_{[0, s)}(y)$. As usual, $Y_t$ is the reflected Brownian motion in $[0, \infty)$. Informally speaking, $X_t$ evolves as $r^{-1} Y_t$ when $Y_t < s$ and stays constant when $Y_t > s$.

The solution of the corresponding ODE is
\formula{
 \ph_\xi(y) & = e^{-i r^{-1} \xi y} \biggl( 1 + \frac{i r^{-1} \xi y}{1 - i r \xi} \biggr)
}
for $y \in [0, s]$, and $\ph_\xi(y) = \ph_\xi(s)$ for $y \in (s, \infty)$. In particular,
\formula{
 \ex^s \bigl[ e^{i \xi X(T_0)} \bigr] & = \ph_\xi(s) = \frac{e^{-i r \xi}}{1 - i r \xi} ,
}
that is, under $\pr^s$, $X(T_0)$ has a centred exponential distribution with parameter $r$ (reflected if $r < 0$).

Note that $X_t$ remains constant on every time interval where $Y_t > s$. Hence, we can erase the fragments of paths of $(X_t, Y_t)$ where $Y_t > s$ without affecting the distribution of $X(T_0)$. Doing so, we obtain an irregular shift-invariant diffusion in $\R \times [0, s]$, reflected both at $\R \times \{0\}$ and at $\R \times \{s\}$. Thus, the centred exponential or reflected exponential distribution with parameter $\lvert r \rvert$ is the Poisson kernel of an irregular shift-invariant diffusion in $\R \times [0, r^2]$, reflected at $\R \times \{0\}$ and at $\R \times \{r^2\}$; the former one, however, is irrelevant for the distribution of $X(T_0)$.

Let us call this diffusion the \emph{exponential layer} or the \emph{reflected exponential layer.}

We record estimates needed for the next step. After erasing the excursions above $s$, and up to $T_0$, $(X_t, Y_t)$ has the Skorokhod representation
\formula*[eq:layer:bm]{
 X_t & = r^{-1} B_t , \\
 Y_t & = s + B_t - K_t ,
}
where $B_t$ is the standard Brownian motion and
\formula{
 K_t & = \max\bigl\{0, \sup\{ B_s : s \in [0, t]\} \bigr\} .
}
For $u(y) = y (2 s - y)$, we have $u(0) = 0$, $u'(s) = 0$ and $\tfrac{1}{2} u''(y) = -1$ for $y \in (0, s)$. Itô's formula, stopped at $\min\{t, T_0\}$, gives
\formula{
 \ex^s \bigl[ \min\{t, T_0\} \bigr] & = s^2 - \ex^s \bigl[u(Y_{\min\{t, T_0\}}) \bigr] \le s^2 .
}
Thus, $T_0 < \infty$ almost surely, and bounded convergence gives
\formula*[eq:layer:bounds]{
 \ex^s \bigl[ T_0 \bigr] & = s^2 = r^4 , \\
 \ex^s \bigl[ \sup \{\lvert X_t \rvert^2 : t \in [0, T_0]\} \bigr] & \le \frac{4}{r^2} \, \ex^s \bigl[ T_0 \bigr] = 4 r^2 ,
}
where the second inequality follows from Doob's inequality for the stopped Brownian motion.

\subsection{Stacking the layers}

Suppose that a probability measure $h$ is a Pólya frequency function, and $p, q, r_n$ are as in~\eqref{eq:pff}. We now stack the Gaussian layer with height $q$ (if $q > 0$) above the finite or infinite sequence of exponential or reflected exponential layers of height $r_n^2$. All these layers are separated by interfaces that only allow $Y_t$ to move downwards. At an interface the process follows the law of the layer below it; that layer reflects at its upper boundary.

For a rigorous definition, we consider the Gaussian layer $(X_t^{(0)}, Y_t^{(0)})$ with height $q$ (or a trivial process if $q = 0$), and a sequence of exponential or reflected exponential layers $(X_t^{(n)}, Y_t^{(n)})$ with heights $r_n^2$. We denote by $\pr^{(p, y_0)}$ the probability under which the above processes are independent, $(X_t^{(0)}, Y_t^{(0)})$ is started at $(p, q)$, and $(X_t^{(n)}, Y_t^{(n)})$ is started at $(0, r_n^2)$. By $T_0^{(n)}$ we denote the boundary hitting time of $(X_t^{(n)}, Y_t^{(n)})$.

Let
\formula{
 y_0 & = q + \sum_{k \ge 1} r_k^2 , &
 y_n & = \sum_{k \ge n} r_k^2 = y_0 - q - r_1^2 - r_2^2 - \ldots - r_{n - 1}^2
}
for $n \ge 1$. We place the Gaussian layer at $\R \times [y_1, y_0]$, and the $n$th exponential layer at $\R \times [y_{n + 1}, y_n]$. More precisely, we define the intransitive diffusion $(X_t, Y_t)$ as follows. Denote
\formula{
 \tau_n & = T_0^{(0)} + T_0^{(1)} + \ldots + T_0^{(n)} .
}
For $t \in [0, \tau_0]$, we set
\formula{
 X_t & = X_t^{(0)}, & Y_t & = y_1 + Y_t^{(0)} .
}
Then, if $(X_t, Y_t)$ has already been defined for $t \in [0, \tau_{n - 1}]$, we let
\formula{
 X(\tau_{n - 1} + t) & = X(\tau_{n - 1}) + X^{(n)}_t , \\
 Y(\tau_{n - 1} + t) & = Y(\tau_{n - 1}) - r_n^2 + Y^{(n)}_t
}
for $t \in (0, T_0^{(n)}]$. Observe that $Y(\tau_n) = y_{n + 1}$ and
\formula{
 X(\tau_n) & = X^{(0)}(T_0^{(0)}) + X^{(1)}(T_0^{(1)}) + \ldots + X^{(n)}(T_0^{(n)}) .
}
Finally, to complete the definition of $(X_t, Y_t)$, we assume that it stays at its final position after $\sup \tau_n$.

If the number of exponential layers is finite, it is straightforward to check that $(X_t, Y_t)$ is a possibly intransitive and irregular shift-invariant diffusion in $\R \times [0, y_0]$, and
\formula{
 X(T_0) & = X^{(0)}(T_0^{(0)}) + \sum_{n \ge 1} X^{(n)}(T_0^{(n)}) .
}
The case of infinite number of layers requires more care.

By~\eqref{eq:layer:bounds},
\formula{
 \ex^{p, y_0} \biggl[ T_0^{(0)} + \sum_{n \ge 1} T_0^{(n)} \biggr] & = q + \sum_{n \ge 1} r_n^4 < \infty ,
}
because $\sum_{n \ge 1} r_n^2 < \infty$. Thus, the partial sums $\tau_n$ of the crossing times $T_0^{(n)}$ converge to a finite limit $\tau_\infty$ almost surely. The independent hitting positions $X^{(n)}(T_0^{(n)})$, $n \ge 1$, have mean $0$ and variances $r_n^2$. Their series therefore converges almost surely and in $L^2$. Consequently, we can define
\formula{
 X_{\tau_\infty + t} & = X^{(0)}(T_0^{(0)}) + \sum_{n \ge 1} X^{(n)}(T_0^{(n)}) , \\
 Y_{\tau_\infty + t} & = 0
}
for $t \ge 0$.

To verify continuity at time $\tau_\infty$, define
\formula{
 M_n & = \sup \{ \lvert X_t^{(n)} \rvert : t \in [0, T_0^{(n)}] \} .
}
By~\eqref{eq:layer:bounds},
\formula{
 \ex^{p, y_0} \biggl[ \sum_{n \ge 1} M_n^2 \biggr] & \le 4 \sum_{n \ge 1} r_n^2 < \infty .
}
Hence, $M_n$ tends to zero almost surely. Together with convergence of $X(\tau_n)$, this proves convergence of the entire horizontal path at the accumulation time. The vertical coordinate lies in intervals whose endpoints decrease to $0$, so its path is continuous there as well. In particular, $T_0 = \tau_\infty$ almost surely.

We use the same construction to define $(X_t, Y_t)$ started at an arbitrary point $(x, y)$: begin in the layer containing that point and concatenate the processes in the remaining lower layers. Each individual layer is strong Markov, and the stopping times $\tau_n$ mark passages through the one-way interfaces between subsequent layers. It is a routine exercise to verify that the concatenated process $(X_t, Y_t)$ also has the strong Markov property. Horizontal translation invariance is immediate from the construction.

Finally,
\formula{
 X(T_0) & = X^{(0)}(T_0^{(0)}) + \sum_{n \ge 1} X^{(n)}(T_0^{(n)}) .
}
Independence and dominated convergence show that the characteristic function of $X(T_0)$ is given by~\eqref{eq:pff}.

Needless to say, if $q = 0$, we omit the Gaussian layer, while if there are no exponential terms in~\eqref{eq:pff}, we do not stack any exponential layer. Finally, we ignore the case when both $q = 0$ and there are no exponential terms, for then $h = \delta_p$, which can be trivially obtained using the uniform downward motion in $\R \times [0, 1]$, started at $(p, 1)$. This proves the following result.

\begin{proposition}
\label{prop:pff}
Every probability measure which is a Pólya frequency function is the hitting distribution of an intransitive irregular shift-invariant diffusion in $\R \times [0, R]$, where $R \in (0, \infty)$, started at the top boundary $\R \times \{R\}$.
\end{proposition}

The diffusion constructed above will be called the \emph{Pólya frequency function layer.}

%
%

\section{Proof of the main result}

With all the tools developed so far, the proof is almost immediate.

\begin{proof}[Proof of Theorem~\ref{thm:main}]
By Corollary~1.9 in~\cite{ks22}, $f$ is the convolution of an $\amcm$ measure $g$ and an integrable Pólya frequency function $h$. Since $f$ has mass at most $1$, both factors have a finite positive mass, and we may normalise both of them so that $h(\R) = 1$ and $g(\R) \le 1$.

By Proposition~\ref{prop:amcm} $g$ is the hitting distribution of a regular transitive shift-invariant diffusion in $\R \times [0, R_f)$, started at $(0, 1)$; here $R_f \in (0, \infty]$. Proposition~\ref{prop:pff} states that $h$ is the hitting distribution of an intransitive shift-invariant diffusion in $\R \times [0, R_g]$, started at $(x, R_g)$; with $R_g \in (0, \infty)$ and $x \in \R$.

We now stack the $\amcm$ layer corresponding to $g$ atop the Pólya frequency function layer corresponding to $h$, with a one-way interface that $Y_t$ can only cross downwards. As in the proof of Proposition~\ref{prop:pff}, we obtain an intransitive shift-invariant diffusion in $\R \times [0, R_g + R_f)$, whose hitting distribution when started at $(x, R_g + 1)$ is the convolution $g * h = f$.

Finally, if $R_f$ is finite, we perform a nonlinear change of variables from $[0, R_g + R_f)$ to $[0, \infty)$ to obtain an intransitive shift-invariant diffusion in $\R \times [0, \infty)$ with the same hitting distribution $f$.
\end{proof}

%
%

\section*{}

\subsection*{Acknowledgements}

I thank Dariusz Kosz for a stimulating discussion after a seminar talk on the results of~\cite{k26}. It was his question that led to the construction described in this note.

I also thank Jacek Wszoła for sharing with me his preliminary work on the discrete analogue and for inspiring discussions.

The large language model \emph{GPT-6 Astra} (OpenAI) was used to assist with reviewing the mathematical arguments and typesetting the document.

This research was funded in whole or in part by National Science Centre, Poland, grant number 2023/49/B/ST1/04303. For the purpose of Open Access, the author has applied a CC-BY public copyright licence to any Author Accepted Manuscript (AAM) version arising from this submission.

%
%

%
%


\begin{thebibliography}{00}

\bibitem{k19}
\textsc{Mateusz Kwaśnicki,}
\emph{Fluctuation theory for Lévy processes with completely monotone jumps.}
Electron. J.~Probab. 24 (2019), no.~40: 1--40.
\doi{10.1214/19-EJP300}

\bibitem{k20}
\textsc{Mateusz Kwaśnicki,}
\emph{A new class of bell-shaped functions.}
Trans. Amer. Math. Soc. 373(4) (2020): 2255--2280.
\doi{10.1090/tran/7825}

\bibitem{k22}
\textsc{Mateusz Kwaśnicki,}
\emph{Harmonic extension technique for non-symmetric operators with completely monotone kernels.}
Calc. Var. Partial. Differ. Equ. 61 (2022), no. 202: 1--40.
\doi{10.1007/s00526-022-02308-2}

\bibitem{k23}
\textsc{Mateusz Kwaśnicki,}
\emph{Boundary traces of shift-invariant diffusions in half-plane.}
Ann. Inst. Henri Poincaré Probab. Statist. 59(1) (2023): 411--436.
\doi{10.1214/22-AIHP1250}

\bibitem{k26}
\textsc{Mateusz Kwaśnicki,}
\textit{Poisson kernels on the half-plane are bell-shaped.}
Bull. London Math. Soc. 58 (2026), no.~e70303: 1--21.
\doi{10.1112/blms.70303}

\bibitem{ks22}
\textsc{Mateusz Kwaśnicki, Thomas Simon,}
\emph{Characterisation of the class of bell-shaped functions.}
Math. Zeitschrift 301(3) (2022): 2659--2683.
\doi{10.1007/s00209-022-02997-7}

\end{thebibliography}
\end{document}